\documentclass[a4paper,numbook,babel,final,envcountsame,11pt]{article}

\usepackage{amssymb}
\usepackage{amsthm}
\usepackage{amsmath}
\usepackage{graphicx}
\usepackage{array,hhline}
\newtheorem{lemma}{Lemma}[section]
\newtheorem{theorem}[lemma]{Theorem}
\newtheorem{proposition}[lemma]{Proposition}
\newtheorem{conjecture}[lemma]{Conjecture}
\newtheorem{corollary}[lemma]{Corollary}
\theoremstyle{definition}
\newtheorem{definition}[lemma]{Definition}
\newtheorem{remark}[lemma]{Remark}

\numberwithin{equation}{section}
\numberwithin{figure}{section}

\newcommand{\Nset}{\mathcal{N}}

\newcommand{\Xset}{\mathcal{X}}
\newcommand{\Yset}{\mathcal{Y}}

\begin{document}

\title{\huge On a new property of $n$-poised and $GC_n$ sets}         % Enter your title between curly braces
\author{Vahagn Bayramyan, Hakop Hakopian}        % Enter your name between curly braces
\date{}          % Enter your date or \today between curly braces

\maketitle

%%%%%%%%%%%%%%%%%%%%%%%%%%%%%%%%%%%%%%%%%%%%%%%%%%%%%%%%%%%%%%%%%%%%%%%%
\begin{abstract}
In this paper we consider n-poised planar node sets, as well
as more special ones, called $GC_n$-sets. For these sets all
$n$-fundamental polynomials are products of n linear factors as it
always takes place in the univariate case. A line ${\ell}$ is called
$k$-node line for a node set $\mathcal X$ if it passes through
exactly $k$ nodes. An $(n+1)$-node line is called maximal line. In
1982 M. Gasca and J. I. Maeztu conjectured that every $GC_n$-set
possesses necessarily a maximal line. Till now the conjecture is
confirmed to be true for $n \le 5$. It is well-known that any
maximal line $M$ of $\mathcal X$ is used by each node in $\mathcal
X\setminus M,$ meaning that it is a factor of the fundamental polynomial
of each node. In this paper we prove, in particular, that if the
Gasca-Maeztu conjecture is true then any $n$-node line of $GC_n$-set
$\Xset$ is used either by exactly $\binom{n}{2}$ nodes or by exactly
$\binom{n-1}{2}$ nodes. We prove also similar statements concerning
$n$-node or $(n-1)$-node lines in more general $n$-poised sets. This
is a new phenomenon in $n$-poised and $GC_n$ sets. At the end we
present a conjecture concerning any $k$-node line.
\end{abstract}


\section{Introduction and background\label{sec:intro}}
Let $\Pi_n$ be the space of bivariate polynomials of total degree
at most $n:$
\begin{equation*}
\Pi_n=\left\{\sum_{i+j\leq{n}}a_{ij}x^iy^j
\right\}.
\end{equation*}
We have that
\begin{equation} \label{N=}
N:=\dim \Pi_n=\binom{n+2}{2}.
\end{equation}
We say that a polynomial $q$ is of degree $k$ if $q\in \Pi_k\setminus \Pi_{k-1}.$
Consider a set of $s$ distinct nodes
\begin{equation*}
{\mathcal X}_s=\{ (x_1, y_1), (x_2, y_2), \dots , (x_s, y_s) \} .
\end{equation*}
The problem of finding a polynomial $p \in \Pi_n$ which satisfies
the conditions
\begin{equation}\label{int cond}
p(x_i, y_i) = c_i, \ \ \quad i = 1, 2, \dots s  ,
\end{equation}
is called interpolation problem.

Let us now describe briefly the content of the paper. We consider
here $n$-poised node sets for which the bivariate interpolation problem is
unisolvent. We pay a special attention to a subclass of these sets
called $GC_n$-sets. In such sets all $n$-fundamental polynomials,
i.e., polynomials of total degree $n$ vanishing at all nodes but
one, are products of $n$ linear factors. Note that this condition
always takes place in the univariate case. A line ${\ell}$ is called
$k$-node line for $\mathcal X$ if it passes through exactly $k$
nodes of $\mathcal X.$ It is easily seen that at most $n+1$ nodes in
an $n$-posed set (and therefore in a $GC_n$ set) can be collinear.
That is why $(n+1)$-node line is called maximal line. In 1982 M.
Gasca and J. I. Maeztu conjectured \cite{GM82} that every $GC_n$-set
possesses necessarily a maximal line. Till now the conjecture is
confirmed to be true for $n \le 5$ (see Subsection \ref{ss:GMconj}).
We say that a node of an $n$-poised or $GC_n$-set uses a line if the
line is a factor in the fundamental polynomial of the node. It is
well-known that any maximal line $M$ of $\mathcal X$ is used by all
nodes in $\mathcal X\setminus M.$ Note that this statement, as well
as the previous one concerning the maximal number of collinear
nodes, follow readily from a well-known and simple fact that a
bivariate polynomial of total degree at most $n$ vanishes on a line
if it vanishes at $n+1$ points in the line (see forthcoming
Proposition \ref{prp:n+1points}). In Section \ref{s:klgcn} we prove
that the subset of nodes of $\Xset$ using a given $k$-node line is
$(k-2)$-independent, meaning that each node of the subset possesses
a fundamental polynomial of total degree not exceeding $k-2.$ In
Sections \ref{s:nlgcn} and \ref{s:pnlgcn} we prove a main result of
this paper. Namely, if the Gasca-Maeztu conjecture is true then any
$n$-node line of a $GC_n$-set $\Xset$ is used either by exactly
$\binom{n}{2}$ nodes or by exactly $\binom{n-1}{2}$ nodes. In
Sections \ref{s:nlnp} and \ref{s:n-1lnp} similar statements are
proved concerning $n$-node or $(n-1)$-node lines in $n$-poised sets.
Let us mention that this is a new phenomenon in $n$-poised and
$GC_n$ sets. At the end we present a conjecture concerning any
$k$-node line.

Now let us go to exact definitions and formulations.
\begin{definition}
The set of nodes ${\mathcal X}_s$ is called \emph{$n$-poised } if
for any data $\{c_1, \dots, c_s\}$ there exists a unique polynomial
$p \in \Pi_n$, satisfying the conditions \eqref{int cond}.
\end{definition}
A polynomial $p \in \Pi_n$ is called an \emph{$n$-fundamental
polynomial} for a node $ A = (x_k, y_k) \in {\mathcal X}_s$ if
\begin{equation*}
p(x_i, y_i) = \delta _{i k},\quad  i = 1, \dots , s ,
\end{equation*}
where $\delta$ is the Kronecker symbol. We denote the
$n$-fundamental polynomial of $A \in{\mathcal X}_s$ by $p_A^\star =
p_{A, {\mathcal X}_s}^\star.$ Sometimes we call fundamental also a
polynomial that vanishes at all nodes but one, since it is a nonzero
constant times the fundamental polynomial.

\noindent In view of the uniqueness we get readily that for any
$n$-poised set the degree of each fundamental polynomial equals to
$n.$

A necessary condition of $n$-poisedness of $\Xset_s$ is:
$|{\mathcal X}_s|=s= N.$

The following is a Linear Algebra fact:
\begin{proposition} \label{prp:poised}
The set of nodes ${\mathcal X}_N$ is $n$-poised if and only if the
following implication holds for any polynomial $p \in \Pi_n:$
$$p(x_i,
y_i) = 0, \quad i = 1, \dots , N \Rightarrow p = 0.$$
\end{proposition}

\subsection{$n$-independent and $n$-dependent sets}

Next we introduce an important concept of $n$-dependence of node sets:
\begin{definition}
A set of nodes ${\mathcal X}$ is called \emph{$n$-independent} if
all its nodes have fundamental polynomials. Otherwise, ${\mathcal
X}$ is called \emph{$n$-dependent.}
\end{definition}
\noindent Clearly fundamental polynomials are linearly independent.
Therefore a necessary condition of $n$-independence is $|{\mathcal
X}| \le N.$

Suppose a node set ${\mathcal X}_s$ is $n$-independent. Then by using the Lagrange formula:
\begin{equation*}
p = \sum_{A\in{\mathcal X}_s} c_{A} p_{{A}, {\mathcal X}_s}^\star
\end{equation*}
we obtain a polynomial $p \in \Pi_n$ satisfying the
interpolation conditions \eqref{int cond}.

\noindent Thus we get a simple characterization of $n$-independence:

\noindent A node set ${\mathcal X}_s$ is $n$-independent if and only
if the interpolation problem \eqref{int cond} is
\emph{$n$-solvable,} meaning that for any data $\{c_1, \dots , c_s
\}$ there exists a (not necessarily unique) polynomial $p \in \Pi_n$
satisfying the conditions \eqref{int cond}.

Now suppose that ${\mathcal X}_s$ is $n$-dependent. Then some node
$(x_{i_0}, y_{i_0})$, does not possess an $n$-fundamental
polynomial. This means that the following implication holds for any
polynomial $p \in \Pi_n:$
$$p(x_i, y_i) = 0, \quad i \in \{1, \dots , s\} \setminus\{i_0\} \Rightarrow p(x_{i_0}, y_{i_0}) =
0.$$

In this paper we will deal frequently with a stronger version of $n$-dependence:

\begin{definition}
A set of nodes ${\mathcal X}$ is called \emph{essentially
$n$-dependent} if none of its nodes possesses a fundamental
polynomial.
\end{definition}

Below, and frequently in the sequel, we use same the notation for a polynomial $q\in\Pi_k$ and the curve described by the
equation $q(x,y)=0$.
\begin{remark}\label{rem:essdep}
Suppose a set of nodes ${\mathcal X}$ is essentially $n$-dependent
and $q\in \Pi_k, \ k\le n,$ is a curve. Then we have that the subset
$\Xset':=\Xset\setminus q$ is essentially $(n-k)$-dependent,
provided that $\Xset'\neq\emptyset.$
\end{remark}
\noindent Indeed, suppose conversely that a node $A\in \Xset'$ has
an $(n-k)$-fundamental polynomial $r\in\Pi_{n-k}.$ Then the
polynomial $qr\in\Pi_n$ is an $n$ fundamental polynomial of the node $A$ in
$\Xset,$ which contradicts our assumption.

\begin{definition}
Given an $n$-poised set $ {\mathcal X}$, we say that a node
$A\in{\mathcal X}$ \emph{uses a curve $q\in \Pi_k$,} if $q$ divides
the fundamental polynomial $p_{A, {\mathcal X}}^\star :$
\begin{equation*}
  p_{A,\Xset}^\star = q r, \quad \text{where} \quad r\in\Pi_{n-k}.
\end{equation*}
\end{definition}

The following proposition is well-known (see e.g. \cite{HJZ09a}
Proposition 1.3):
\begin{proposition}\label{prp:n+1points}
Suppose that  $\ell$ is a line. Then for any polynomial $p \in
\Pi_n$ vanishing at $n+1$ points of $\ell$ we have
\begin{equation*}
p = \ell  r  ,\quad \text{where} \quad r\in\Pi_{n-1}.
\end{equation*}
\end{proposition}
Evidently, this implies that any set of $n+2$ collinear nodes is
essentially $n$-dependent. We also obtain from Proposition
\ref{prp:n+1points}
\begin{corollary}\label{cor:n+1points} The following hold for any $n$-poised node set $\Xset:$
\begin{enumerate}
\item
At most $n+1$ nodes of $\Xset$ can be collinear;
\item
A line $\ell$
containing $n+1$ nodes of $\Xset$ is used by all the nodes in ${\mathcal
X}\setminus \ell$.
\end{enumerate}
\end{corollary}
In view of this a line $\ell$ containing $n+1$
nodes of an $n$-poised set ${\mathcal X}$ is called a maximal line
(see \cite{dB07}).

One can verify readily the following two properties of maximal lines
of $n$-poised set $\Xset:$
\begin{enumerate}
\item
Any two maximal lines of $\Xset$ intersect necessarily at a node of $\Xset;$
\item
Three maximal lines of $\Xset$ cannot meet in one node.
\end{enumerate}

Thus, in view of \eqref{N=}, there are no $n$-poised sets with more than $n+2$ maximal
lines.

\subsection{Some results on $n$-independence}

Let us start with the following simple but important result of Severi \cite{S}:
\begin{theorem}[Severi]\label{thm:severi}
Any node set $\Xset$ consisting of at most $n+1$ nodes is
$n$-independent.
\end{theorem}

\noindent Next we consider node sets consisting of at most $2n+1$
nodes:
\begin{proposition}[\cite{EGH}]\label{prp:2n+1}
Any node set $\Xset$ consisting of at most $2n+1$ nodes is
$n$-dependent if and only if $n+2$ nodes are collinear.
\end{proposition}

\noindent
For a generalization of above two results for multiple nodes see \cite{S} and \cite{H00}, respectively.

\noindent The third result in this series is the following
\begin{proposition}\label{prp:3n-1}
Any node set $\Xset$ consisting of at most $3n-1$ nodes is
$n$-dependent if and only if at least one of the following holds.
\vspace{-2mm}
\begin{enumerate}
\setlength{\itemsep}{0mm}
\item
$n+2$ nodes are collinear,
\item
$2n+2$ nodes belong to a conic (possibly reducible).
\end{enumerate}
\end{proposition}

Let us mention that this result as well as the two previous results
are special cases of the following

\begin{theorem}[\cite{HM12}, Thm.~5.1]\label{thm:3n}
Any node set $\Xset$ consisting of at most $3\,n$ nodes is
$n$-dependent if and only if at least one of the following holds.
\vspace{-2mm}
\begin{enumerate}
\setlength{\itemsep}{0mm}
\item
$n+2$ nodes are collinear,
\item
$2\,n+2$ nodes belong to a conic,
\item
$|\Xset| = 3\,n,$ and there is a cubic $\gamma\in\Pi_3$ and an
algebraic curve $\sigma\in\Pi_n$ such that $\Xset = \gamma \cap
\sigma$\,.
\end{enumerate}
\end{theorem}

%%%%%%%%%%%%%%%%%%%%%%%%%%%%%%%%%%%%%%%%%%%%%%%%%%%%%%%%%%%

\subsection{$GC_n$ sets and the Gasca-Maeztu conjecture \label{ss:GMconj}}

Let us consider a special type of $n$-poised sets whose
$n$-fundamental polynomials are products of $n$ linear factors as it
always takes place in the univariate case:
\begin{definition}[Chung, Yao \cite{CY77}]
An n-poised set ${\mathcal X}$ is called \emph{$GC_n$-set} if  the
$n$-fundamental polynomial of each node $A\in{\mathcal X}$ is a
product of $n$ linear factors.
\end{definition}
\noindent In other words, $GC_n$ sets are the sets each node of
which uses exactly $n$ lines.

Now we are in a position to present the Gasca-Maeztu conjecture,
called briefly also GM conjecture:

\begin{conjecture}[Gasca, Maeztu, \cite{GM82}]\label{conj:GM}
Any $GC_n$-set contains $n+1$ collinear nodes.
\end{conjecture}

\noindent Thus the GM conjecture states that any $GC_n$ set
possesses a maximal line.

\noindent So far, this conjecture has been verified for the degrees
$n\leq 5.$ For $n=2$, the conjecture is evidently true. The case
$n=3$ is not hard to prove. The case $n=4$ was proved for the first
time by J.~R.~Busch in 1990~\cite{B90}. Since then, four other
proofs have appeared for this case: \cite{CG01, HJZ09b,BHT}, and
\cite{ST}. In our opinion the last one is the simplest and the
shortest one. The case $n=5$ was proved recently in \cite{HJZ14}.

Notice that if a line $M$ is maximal then the set $\Xset\setminus M$
is $(n-1)$-poised. Moreover, if $\Xset$ is a $GC_n$-set then
$\Xset\setminus M$ is a $GC_{n-1}$-set.

\noindent For a generalization of the Gasca-Maeztu conjecture to
maximal curves see \cite{HR}.

In the sequel we will make use of the following important result of
Carnicer and Gasca concerning the GM conjecture:

\begin{theorem}[Carnicer, Gasca, \cite{CG03}]\label{thm:CG}
If the Gasca-Maeztu conjecture is true for all $k\leq n$, then any $GC_n$-set possesses at least
three maximal lines.
\end{theorem}

\noindent In view of this one gets readily that each node of $\Xset$
uses at least one maximal line.

%%%%%%%%%%%%%%%%%%%%%%%%%%%%%%%%%%%%%%%%%%%%%%%%%%%%%%%%%%%%%%%%%%%%%%%%
\subsection{Some examples of $n$-poised and $GC_n$ sets\label{ssec:methods}}

We will consider $3$ well-known constructions: \emph{The
Berzolari-Radon construction} \cite{B14,R48}, \emph{the Chung-Yao
construction} \cite{CY77} (called also \emph{Chung-Yao natural
lattice)}, and the \emph{principal lattice }. The first construction
gives examples of $n$-poised sets, while the remaining two give
examples of $GC_n$-sets. Let us mention that both the Chung-Yao
natural lattice and the principal lattice are special cases of the
Berzolari-Radon construction.

\noindent Note that Lagrange and Newton formulas for these
constructions can be found in \cite{GM82} and~\cite{J83}.

%%%%%%%%%%%%%%%%%%%%%%%%%%%%%%%%%%%%%%%%%%%%%%%%%%
\subsubsection*{The Berzolari-Radon construction}

A set $\Xset$ containing $N=1+2+\cdots+(n+1)$ nodes is called
Berzolari-Radon set if there are $n+1$ lines:
$\ell_1,\ldots,\ell_{n+1}$ such that the sets $\ell_1,\
\ell_2\setminus \ell_1, \ \ell_3\setminus(\ell_1\cup\ell_2),\ldots,
\ell_{n+1} \setminus \bigl( \bigcup_{i=1}^{n} \ell_i \bigr)$ contain
exactly $(n+1), n, (n-1),\ldots, 1$ nodes, respectively. The
Berzolari-Radon set is $n$-poised.

\noindent It is worth noting that the Gasca-Maeztu conjecture is
equivalent to the statement that every $GC_n$-set is a
Berzolari-Radon set.

%%%%%%%%%%%%%%%%%%%%%%%%%%%%%%%%%%%%%%%%%%%%%%%%%%%%%%%%%%%%%%%%%%%
\subsubsection*{The Chung-Yao construction \label{ss:CY}}

Consider $n+2$ lines: $\ell_1,\ldots,\ell_{n+2},$ such that no two
lines are parallel, and no three lines intersect in one point. Then
the set $\Xset$ of intersection points of these lines is called
Chung-Yao set. Notice that $|\Xset|=\binom{n+2}{2}.$ Each fixed node
here is lying in exactly $2$ lines, and does not belong to the
remaining $n$ lines. Moreover, the product of these $n$ lines gives
the fundamental polynomial of the fixed node. Thus $\Xset$ is
$GC_n$-set.

\noindent Note that this construction can be characterized by the
fact that all the given $n+2$ lines are maximal. As it was mentioned
earlier, there are no $n$-poised sets with more maximal lines.

\subsubsection*{The principal lattice}

The principal lattice is the following set (or an affine image of it)
\begin{equation*}
  \Xset = \bigl\{ (i,j) \in{\mathbb Z}_+^2 : i+j\leq n \bigr\}.
\end{equation*}
Notice that the fundamental polynomial of the node $(i,j)$ here uses
$i$ vertical lines: $x=k,\ k=0,\ldots, i-1,$ $j$ horizontal lines:
$y=k,\ k=0,\ldots, j-1$ and $n-i-j$ lines with slope $-1:$ $x+y=k,\
k=i+j+1,\ldots, n.$ Thus $\Xset$ is $GC_n$-set.

\noindent Note that this lattice possesses just three maximal lines,
namely the lines $x=0$, $y=0$, and $x+y=n$. Note that, according to
Theorem \ref{thm:CG}, there are no $n$-poised sets with less maximal
lines, provided that the Gasca-Maeztu conjecture is true.

\subsection{Maximal curves and the sets $\Nset_q$ and $\Xset_\ell$}

Let us start with a generalization of
Proposition \ref{prp:n+1points} for algebraic curves
of higher degree. First set for $k\leq n$
\begin{equation*}
  d(n,k) := \dim\Pi_n - \dim\Pi_{n-k} = \tfrac{1}{2} \, k \,
  (2\,n\,{+}\,3\,{-}\,k).
\end{equation*}

\begin{proposition}[Rafayelyan, \cite{R11}, Prop.~3.1]\label{prp:ind1}
Let $q$ be an algebraic curve of degree $k\le n$ without multiple
components. Then the following hold.
\vspace{-2mm}
\begin{enumerate}
\setlength{\itemsep}{0mm}
\item
Any subset of $q$ consisting of more than $d(n,k)$ nodes is
$n$-dependent.
\item
A subset $\Xset \subset q$ consisting of $d(n,k)$ nodes is
$n$-independent if and only if the following implication holds for
any polynomial $p\in\Pi_n:$
\begin{equation}\label{eq:xgam1}
   p\big\vert_\Xset = 0 \quad\implies\quad p = q r
  \quad\text{for some } r\in\Pi_{n-k}.
\end{equation}
\end{enumerate}
\end{proposition}

\noindent Let us mention that a special case of (i), when $q$
factors into linear factors, is due to Carnicer and Gasca,
\cite{CG00}.

As in the case of lines (see Corollary \ref{cor:n+1points})
we get readily from here

\begin{corollary} The following hold for any $n$-poised node set $\Xset:$
\begin{enumerate}
\item
At most $d(n,k)$ nodes of $\Xset$ can lie in a curve of degree $k;$
\item
A curve of degree $k\le n$ without multiple components containing
$d(n,k)$ nodes of $\Xset$ is used by all the nodes in ${\mathcal
X}\setminus q$.
\end{enumerate}
\end{corollary}
Next we bring a generalization of the concept of a maximal line (see \cite{R11}):

\begin{definition}\label{def:mc}
A curve of degree $k\le n$ without multiple components passing
through $d(n,k)$ nodes of an $n$-poised set $\Xset$ is called a
\emph{maximal curve} for $\Xset$.
\end{definition}

\noindent Thus maximal line, conic, and cubic pass through $n+1$,
$2n+1$, and $3n$ nodes of $\Xset$, respectively.

Below, for an $n$-posed set $\Xset,$ line $\ell$ and an algebraic curve $q,$ we define
important sets $\Xset_\ell$ and $\Nset_q,$ which will be used frequently in the sequel.

\begin{definition} \label{def:Nq}
Let $\Xset$ be an $n$-poised set $\ell$ be a line and  $q$ be an algebraic curve
without multiple factors. Then
\begin{enumerate}
\item
$\Xset_\ell$ is the subset of nodes of
$\Xset$ which use the line $\ell;$
\item
$\Nset_q$ is the subset of nodes of
$\Xset$ which do not use the curve $q$ and are not lying in $q.$
\end{enumerate}
\end{definition}

Next let us bring a characterization of maximal curves:
\begin{proposition}[Rafayelyan, \cite{R11}, Prop.~3.3]\label{prp:mc}
Let $\Xset$ be an $n$-poised set and  $q$ be an algebraic curve of degree $k\le n$ without multiple factors.
Then the following statements are equivalent:
\begin{enumerate}
\item
The curve $q$ is maximal for $\Xset$;
\item
All the nodes in $\Yset:=\Xset\setminus q$ use the curve $q,$ i.e., $\Nset_q=\emptyset$;
\item
The set $\Yset$ is $(n-k)$-poised. Moreover, if $\Xset$ is a
$GC_n$-set then $\Yset$ is a $GC_{n-k}$-set.
\end{enumerate}
\end{proposition}

\noindent Thus $\Nset_q=\emptyset$ means that $q$ is a maximal
curve. The following result concerns the case when
$\Nset_q\neq\emptyset.$
\begin{proposition}[Rafayelyan, \cite{R11}]\label{prp:Nq}
Let $\Xset$ be an $n$-poised set and  $q$ be an algebraic curve of degree $k\le n$ without multiple factors.
Then the set $\Nset_q$ is essentially $(n-k)$-dependent, provided that it is not empty.
\end{proposition}

It is worth mentioning that the special case $k=1$ of above two results, where $q$ is a line
is due to Carnicer and Gasca \cite{CG01}. Note also that the case when $q$ is a product of $k$ lines
is proved in \cite{HJZ09b}.

\begin{proposition}[\cite{BHT15}]\label{prp:2nl}
Let $\Xset$ be an $n$-poised set. Then there is at most one
algebraic curve of degree $n-1$ passing through $N-4$ nodes of $\Xset.$
\end{proposition}
From here one gets readily for any $n$-poised set $\Xset$ (see \cite{B15}):

\begin{equation}\label{2nodel} |\Xset_\ell|\le 1\ \ \hbox{if}\ \ {\ell}\ \hbox{is a $2$-node line.}
\end{equation}
For a generalization of these results for curves of arbitrary degree and $3$-node lines see \cite{HT1, HT2}.
Let us mention that the statement \eqref{2nodel} for $GC_n$ sets has already been shown in \cite{CG03}.

In the sequel we will use frequently the following $2$ lemmas from
\cite{CG03}.  Let us mention that the the second lemma is used in a
proof there and is not explicitly formulated. For the sake of
completeness we bring proofs here.

\begin{lemma}[Carnicer, Gasca, \cite{CG03}]\label{lem:CG1}
Let $\Xset$ be an $n$-poised set and ${\ell}$ be a line. Suppose also
that there is a maximal line $M_0$ such that $M_0 \cap {\ell} \notin
\Xset.$ Then we have that
\begin{equation*}
\Xset_{\ell} = {(\Xset \setminus M_0)}_{\ell}.
\end{equation*}
If in addition ${\ell}$ is an $n$-node line then we have that
\begin{equation*}
\Xset_{\ell} = \Xset \setminus ({\ell} \cup M_0)\ \hbox{and therefore}\
|\Xset_{\ell}| = \binom{n}{2}.
\end{equation*}
\end{lemma}
\begin{proof}Suppose conversely that a node $A\in M_0$ uses ${\ell}:$
$$ p^\star_A= {\ell}q,\quad q\in\Pi_{n-1}.$$
Notice that $q$ vanishes at the $n$ nodes in $M_0\setminus \{A\}$ Thus, in view of Proposition \ref{prp:n+1points}, we have that $q$ and hence $p^\star_A$ vanishes on $M_0.$
In particular $p$ vanishes at $A,$ which is a contradiction.

Now assume that ${\ell}$ is an $n$-node line and $A\notin {\ell}\cup M_0.$ Then we have that
$$ p^\star_A= M_0q,\quad q\in\Pi_{n-1}.$$
Notice that $q$ vanishes at the $n$ nodes of ${\ell}.$ Thus, in view of
Proposition \ref{prp:n+1points}, we have that $q={\ell}r,\
r\in\Pi_{n-2}.$ Thus we obtain that $p^\star_A=M_0{\ell}r,$ i.e., $A$
uses the line ${\ell}.$
\end{proof}

\begin{lemma}[Carnicer, Gasca, \cite{CG03}]\label{lem:CG2}
Let $\Xset$ be an $n$-poised set and ${\ell}$ be a line. Suppose also
that there are two maximal lines $M', M''$ such that $M' \cap M''
\cap {\ell} \in \Xset.$ Then we have that
\begin{equation*}
\Xset_{\ell} = {(\Xset \setminus (M' \cup M''))}_{\ell}.
\end{equation*}
If in addition ${\ell}$ is an $n$-node line then we have that
\begin{equation*}
\Xset_{\ell} = \Xset \setminus ({\ell} \cup M' \cup M'')\ \hbox{and
therefore}\ |\Xset_{\ell}| = \binom{n-1}{2}.
\end{equation*}
\begin{proof} Suppose that a node $A\in (M'\cup M'')\setminus {\ell}$ uses ${\ell}:$
$$ p^\star_A= {\ell}q,\quad q\in\Pi_{n-1}.$$
Suppose, for example, $A\in M'.$ Then notice that $q$ vanishes at
the $n$ nodes of $M''\setminus {\ell}.$ Thus, in view of Proposition
\ref{prp:n+1points}, we have that $q=M''r,\ r\in \Pi_{n-2}.$ Then
$r$ vanishes on $n-1$ nodes in $M'\setminus {\ell}$ different from $A.$
Thus $r$ and hence also $p^\star_A$ vanishes on whole line $M'$
including $A,$ which is a contradiction.

Now assume that ${\ell}$ is an $n$-node line and $A\notin {\ell}\cup M'\cup
M''.$ Then we have that
$$ p^\star_A= M'M''q,\quad q\in\Pi_{n-2}.$$
Notice that $q$ vanishes at the $n-1$ nodes of ${\ell}$ different
from the node of intersection with the maximal lines. Thus, in view
of Proposition \ref{prp:n+1points}, we have that $q={\ell}r,\quad r\in\Pi_{n-3}. $ Therefore we obtain that
$p^\star_A=M'M''{\ell}r,$ i.e., $A$ uses the line ${\ell}.$
\end{proof}

\end{lemma}

\section{Lines in $n$-poised sets \label{s:nlnp}}

\subsection{On $n$-node lines in $n$-poised sets \label{ss:nlnp}}

Let us start our results with the following
\begin{proposition}\label{prp:linennp}
Let $\Xset$ be an $n$-poised set and ${\ell}$ be a line passing through exactly $n$ nodes of $\Xset.$
Then the following hold:
\begin{enumerate}
\setlength{\itemsep}{0mm}
\item
$|\Xset_{\ell}| \leq \binom{n}{2};$
\item
If $|\Xset_{\ell}| \ge \binom{n-1}{2}+1$ then there is a maximal line
$M_0$ such that $M_0 \cap {\ell} \notin \Xset.$ Moreover, we have that
$\Xset_{\ell}=\Xset\setminus (M_0\cup {\ell}).$ Hence it is an $(n-2)$-poised
set. In particular we have that $|\Xset_{\ell}|=\binom{n}{2};$
\item
If $\binom{n-1}{2}\geq |\Xset_{\ell}| \geq \binom{n-2}{2}+2$, then $|\Xset_{\ell}| = \binom{n-1}{2}.$
Moreover,  $\Xset_{\ell}$ is an $(n-3)$-poised set and there is a conic $\beta\in \Pi_2$ such that $\Xset_{\ell}=\Xset\setminus (\beta \cup {\ell}).$ Furthermore, we have that $\Nset_{\ell} \subset \beta$ and $|(\beta\setminus {\ell})\cap\Xset|=|\Nset_{\ell}|=2n.$
Besides these $2n$ nodes the conic may contain at most one extra node, which necessarily belongs to ${\ell}.$ If
$\beta$ is reducible: $\beta={\ell}_1{\ell}_2$ then we have that $|{\ell}_i\cap(\Xset\setminus \ell)|=n,\ i=1,2.$
\end{enumerate}
\end{proposition}

\begin{proof}
\begin{enumerate}
\item
Assume by way of contradiction that $|\Xset_{\ell}| \ge \binom{n}{2}+1$. Then we obtain
\begin{equation*}
|\Nset_{\ell}| \le \binom{n+2}{2} - \left[\binom{n}{2}+1\right] - n = n.
\end{equation*}
This is a contradiction, since on one hand, in view of Proposition \ref{prp:mc}, the nonempty set $\Nset_{\ell}$ is $(n-1)$-dependent
and on the other hand, in view of Theorem \ref{thm:severi}, it is $(n-1)$-independent.

\item
In this case we have that
\begin{equation*}
|\Nset_{\ell}| \le \binom{n+2}{2} - \left[\binom{n-1}{2}+1\right] - n = 2n-1=2(n-1)+1.
\end{equation*}
Now let us make use of Proposition \ref{prp:2n+1}. Since $\Nset_{\ell}$ is $(n-1)$-dependent, we get that there is a line $M_0$ passing through $n+1$ nodes of $\Nset_{\ell}.$
The line $M_0$ is maximal and therefore cannot pass through any more nodes. Hence we obtain that
$M_0 \cap {\ell} \notin \Xset.$ Thus, in view of Lemma \ref{lem:CG1}, we have that $\Xset_{\ell}$ is an $(n-2)$-poised set.

\item
In this case we have that
\begin{equation*}
|\Nset_{\ell}| \leq \binom{n+2}{2} - \left[\binom{n-2}{2}+2\right] - n = 3n-4 = 3(n-1)-1.
\end{equation*}
\end{enumerate}

\noindent Since the set $\Nset_{\ell}$ is $(n-1)$-dependent, we get from Proposition \ref{prp:3n-1}, that either

a) there is a line $M_0$ passing through $n+1$ nodes in $\Nset_{\ell},$ or

b) there is a conic $\beta \in \Pi_2$ passing through $2n=2(n-1)+2$ nodes in $\Nset_{\ell}.$

Let us start with the case a). We have for the maximal line $M_0$, in the same way as in the case ii),
that $M_0 \cap {\ell} \notin \Xset$ and therefore $|\Xset|=\binom{n}{2}.$ This contradicts our assumption in iii).

In the case b) let us first show that $|\Nset_{\ell}|=2n.$ Indeed, in
view of Proposition \ref{prp:mc}, we have that $\Nset_{\ell}$ is
essentially $(n-1)$-dependent. Then, suppose $\Nset_{\ell},$ besides the nodes in
$\beta$,  contains $t$ nodes outside of it, where $t\le
n-4\ (=3n-4-2n).$ In view of Remark \ref{rem:essdep} these $t$ nodes
must be $(n-3)$-essentially dependent. Therefore, we get from
Theorem \ref{thm:severi} that $t=0.$ Now notice that ${\ell}\beta$ is a maximal cubic since it passes through $3n$ nodes.
The conic $\beta$, besides the
$2n$ nodes, may contain at most $1$ extra node, since the
set $\Xset$ is $n$-independent. But, if the extra node does not belong to ${\ell},$ then the
cubic ${\ell}\beta$ would contain $3n+1$ nodes, which is a contradiction.

Finally assume that the conic is reducible: $\beta={\ell}_1{\ell}_2.$ Then, since
$\Nset_{\ell}$ is $(n-1)$-essentially dependent, we readily get that each
of the lines passes through exactly $n$ nodes from the $2n.$
\end{proof}

\subsection{ On $(n-1)$-node lines in $n$-poised
sets\label{s:n-1lnp}}

\begin{proposition}\label{prp:n-1lnp}
Let $\Xset$ be an $n$-poised set and ${\ell}$ be a line passing through
exactly $n-1$ nodes of $\Xset.$ Assume also that $|\Xset_{\ell}| \geq
\binom{n-2}{2}+3$. \\ Then we have that $\Xset_{\ell}$ is $(n-3)$-poised
set. Hence $|\Xset_{\ell}| = \binom{n-1}{2}$ and $|\Nset_{\ell}|=2n+1.$
Moreover, these $2n+1$ nodes are located in the following way:
\begin{enumerate}
\item
$n+1$ nodes are in a maximal line $M_0$ and
\item
$n$ nodes are in an $n$-node line $M_0'.$
\end{enumerate}
Furthermore, besides these $n$ nodes, the line $M_0'$ may contain at
most one extra node, which necessarily belongs to $M_0.$
\end{proposition}

\begin{proof}
We have that
\begin{equation*}
|\Nset_{\ell}| \leq \binom{n+2}{2} - \left[\binom{n-2}{2}+3\right] -
(n-1) = 3n-4 = 3(n-1)-1.
\end{equation*}
According to Proposition \ref{prp:mc} the set $\Nset_{\ell}$ is essentially
$(n-1)$-dependent. Therefore, in view of Proposition \ref{prp:3n-1},
we have that either
\begin{enumerate}
\item
there is a line $M_0$ passing through $n+1$ nodes of $\Nset_{\ell}$, or
\item
there is a conic $\beta \in \Pi_2$ passing through $2n=2(n-1)+2$
nodes of $\Nset_{\ell}.$
\end{enumerate}

Assume that i) holds. Then, suppose there are $s$ nodes in
$\Nset_{\ell}$ outside the line $M_0,$ where $s\le 2(n-2)-1
(=2n-5=3n-4-n-1).$

Let us verify that $s\neq 0.$ Assume conversely that $s=0.$ Then we
have that any node $A\in \Xset\setminus ({\ell}\cup M_0)$ uses the line
${\ell}$ and maximal line $M_0,$ i.e.,
\begin{equation*}
p^*_A={\ell}M_0q, \qquad q\in\Pi_{n-2}.
\end{equation*}
This means, in view of Proposition \ref{prp:mc} (part i)
$\Leftrightarrow$ ii)), that the conic ${\ell}M_0$ is maximal, which is
contradiction, since it passes through only $2n$ nodes (instead of
$2n+1$ nodes).

Then, in view of Remark \ref{rem:essdep} these $s$ nodes must be
$(n-2)$-essentially dependent. Therefore, by Proposition
\ref{prp:2n+1}, there is a line $M_0'$ passing through $n$ nodes of
$\Nset_{\ell}\setminus M_0.$ Now, suppose there are $t$ nodes in $\Nset_{\ell}$
outside the lines $M_0$ and $M_0',$  where $t:\le
(n-3)-2 \ (=n-5=2n-5-n).$  These $t$ nodes, in view of Remark \ref{rem:essdep}, must be
essentially $(n-3)$-dependent. Thus, we conclude from Theorem
\ref{thm:severi} that $t=0$ and therefore $|\Nset_{\ell}|=2n+1.$

Now, it remains to verify that the case ii) is impossible.

Thus assume that ii) holds. Denote the number of nodes in $\Nset_{\ell}$
outside the conic $\beta$ by $t.$ We have that $t \le (n-2)-2 \
(=n-4=3n-4-2n).$ In view of Remark \ref{rem:essdep} these $t$ nodes
must be $(n-3)$-essentially dependent. Therefore, by Theorem
\ref{thm:severi}, we obtain that $t=0$ and therefore $|\Nset_{\ell}|=2n.$

Now we have that any node $A\in \Xset\setminus ({\ell}\cup\beta)$ uses
the line ${\ell}$. This means
\begin{equation*}
p^*_A={\ell}q, \qquad q\in\Pi_{n-1}.
\end{equation*}
The curve $q$ passes through all the $2n$ nodes in $\beta$. By the
Bezout theorem we conclude that $q$ divides $\beta.$ Indeed, this is
evident when $\beta$ is irreducible. Now assume that $\beta$ is
reducible, i.e., $\beta={\ell}_1{\ell}_2.$  The set $\Nset_\beta$ is
$(n-1)$-essentially dependent. Therefore each line ${\ell}_i, i=1,2,$
passes through exactly $n$ nodes of $\Nset_\beta$ and hence divides
$q.$ Thus we have that $q=\beta r, \ q\in\Pi_{n-3}.$ Finally we get
\begin{equation*}
p^*_A={\ell}\beta r, \qquad r\in\Pi_{n-3}\quad \hbox{for any}\quad  A\in
\Xset\setminus ({\ell}\cup\beta).
\end{equation*}
This means that each node outside ${\ell}$ and $\beta$ uses the reducible
cubic ${\ell}\beta.$ Therefore, by Proposition \ref{prp:mc} (part i)
$\Leftrightarrow$ ii)), the latter curve is maximal, which is
contradiction, since it passes through only $3n-1$ nodes (instead of
$3n$ nodes).
\end{proof}

\begin{corollary}\label{ge}
Let $\Xset$ be an $n$-poised set and ${\ell}$ be a line passing through
exactly $n-1$ nodes of $\Xset.$ Then we have that $|\Xset_{\ell}| \le
\binom{n-1}{2}.$
\end{corollary}

\begin{proof}
Assume by way of contradiction that $|\Xset_{\ell}| \ge
\binom{n-1}{2}+1.$ Notice that
$$\binom{n-1}{2}+1\ge \binom{n-2}{2}+3\ \hbox{if} \ n\ge 4.$$
Now, in view of Proposition \ref{prp:n-1lnp}, we get that $|\Xset_{\ell}|
\le \binom{n-1}{2},$ which contradicts our assumption. It remains to note that
 Corollary in the case $n =3$ is a special case in \eqref{2nodel}.
\end{proof}

\section{Lines in $GC_n$ sets\label{s:klgcn}}

\subsection{On $k$-node lines in $GC_n$ sets\label{ss:klgcn}}

\begin{proposition}\label{prp:linekgcn}
Assume that Conjecture \ref{conj:GM} holds for all degrees up to
$\nu$. Let $\Xset$ be a $GC_n$ set, $n \leq \nu,$ and ${\ell}$ be a line
passing through exactly $k$ nodes of $\Xset.$ Then the set $\Xset_{\ell}$
is $(k-2)$-independent set. Moreover, for each node $A\in \Xset_{\ell}$
there is a $(k-2)$-fundamental polynomial that divides the
$n$-fundamental polynomial of $A$ in $\Xset.$
\end{proposition}

\begin{proof}
First suppose that $k=n+1,$ meaning that ${\ell}$ is a maximal line. Then
we have that $\Xset_{\ell}=\Xset\setminus {\ell}$ and this set is
$GC_{n-1}$-set and hence is $(n-1)$-poised.

In the case when ${\ell}$ is not maximal we will use induction on $n$.
The case $n=2$ is evident (see Subsection \ref{ss:case2}). Suppose
Proposition is true for all degrees less than $n$ and let us prove
it for $n$.

Suppose that there is a maximal line $M_0$ such that $M_0 \cap {\ell}
\notin \Xset.$ Then we get from Lemma \ref{lem:CG1} that $\Xset_{\ell} =
(\Xset_0)_{\ell}$ where $\Xset_0:=\Xset \setminus M_0.$ We have that the
set $\Xset_0$ is $GC_{n-1}$-set and ${\ell}$ passes through exactly $k$
nodes of $\Xset_0.$ Therefore by induction hypothesis for the degree
$n-1$ we get that $\Xset_{\ell}$ is $(k-2)$-independent.

%\begin{figure} %stronger is !ht, strongest H
%\centering
%\includegraphics[scale=0.8]{pic1.png}
%\caption{Four concurrent lines} \label{pic1}
%\end{figure}

Now, in view of Theorem \ref{thm:CG}, consider three maximal lines
for $\Xset$  and denote them by $M_i,\ i=1,2,3.$ It remains to
consider the case when each of these maximal lines intersects ${\ell}$ at
a node of $\Xset.$

We will prove that $\Xset_{\ell}$ is $(k-2)$-independent by finding a
$(k-2)$-fundamental polynomial for each node $A \in \Xset_{\ell}$. Since
$3$ maximal lines intersect each other at $3$ distinct nodes there
is $i_0 \in \{1,2,3\}$ such that $A \notin M_{i_0}$. We have that
the set $\Yset :=\Xset \setminus M_{i_0}$ is $GC_{n-1}$-set and ${\ell}$
passes through exactly $k-1$ nodes of $\Yset.$ Therefore by
induction hypothesis for the degree $n-1$ we get that the set
$\Yset_{\ell}$ is $(k-3)$-independent. Moreover, there is a
$(k-3)$-fundamental polynomial $ p^\star_{A, {\Yset}_{\ell}} \in
\Pi_{k-3}$ which divides $p^\star_{A, \Yset}.$

Now, since $\Xset_{\ell} \subset \Yset \cup M_{i_0},$ we get readily that
the polynomial
\begin{equation*}
M_{i_0}p^\star_{A, {\Yset}_{\ell}}\in \Pi_{k-2}
\end{equation*}
is a fundamental polynomial of $A$ in $\Xset_{\ell}.$ We get also that it
divides the polynomial $p^\star_{A,\Xset}=M_{i_0}p^\star_{A,
\Yset}.$
\end{proof}

Below we bring some simple consequences of the fact that the set
$\Xset_{\ell}$ is $(k-2)$-independent:
\begin{corollary}\label{col:indep}
Assume that the conditions of Proposition \ref{prp:linekgcn} hold.
Then the following hold.
\begin{enumerate}
\setlength{\itemsep}{0mm}
\item
$|\Xset_{\ell}| \leq \binom{k}{2};$
\item
$\Xset_{\ell}$ contains at most $k-1$ collinear nodes;
\item
For any curve $q$ of degree $m \leq k-2$ we have that
\begin{equation*}
|\Xset_{\ell}\cap q| \leq d(k-2,m).
\end{equation*}.
\end{enumerate}
\end{corollary}
Note, that ii) is a special case of iii) when $m=1.$ Let us mention
that i) and ii) were proved in \cite{CG03}, Theorem 4.5.

\bigskip

\subsection{On $n$-node lines in $GC_n$ sets\label{s:nlgcn}}

Next, let us present a main result of this paper:
\begin{theorem}\label{thm:linengcn}
Assume that Conjecture \ref{conj:GM} holds for all degrees up to
$\nu$. Let $\Xset$ be a $GC_n$ set, $n \leq \nu$ and ${\ell}$ be a line
passing through exactly $n$ nodes of the set $\Xset.$ Then we have
that
\begin{equation} \label{2bin} |\Xset_{\ell}| = \binom{n}{2}\quad \hbox{or} \quad \binom{n-1}{2}.
\end{equation}
Also, the following hold:
\begin{enumerate}
\item
If $|\Xset_{\ell}| = \binom{n}{2}$ then there is a maximal line $M_0$ such that $M_0 \cap {\ell} \notin \Xset.$ Moreover,
we have that $\Xset_{\ell}=\Xset\setminus ({\ell}\cup M_0).$ Hence it is a $GC_{n-2}$ set;

\item
If $|\Xset_{\ell}| = \binom{n-1}{2}$ then there are two maximal lines $M', M'',$ such that $M' \cap M'' \cap {\ell} \in \Xset.$
Moreover,
we have that $\Xset_{\ell}=\Xset\setminus ({\ell}\cup M'\cup M'').$ Hence is a $GC_{n-3}$ set.
\end{enumerate}
\end{theorem}
Let us first assume that Theorem is valid and prove the following
\begin{corollary}\label{cor:linengcn}
Assume that the conditions of Theorem \ref{thm:linengcn} take place. Then the following hold for any maximal line $M$ of $\Xset$:
\begin{enumerate}
\item
$|M\cap \Xset_{\ell}| = 0$ if

a) $M \cap {\ell} \notin \Xset$  or if

b) there is another maximal line $M'$ such that $M\cap M'\cap {\ell} \in \Xset;$

\item
$|M\cap \Xset_{\ell}| = s-1\ $ if $\ |\Xset_{\ell}|=\binom{s}{2},$ where
$s=n,n-1,$ \\ for all the remaining maximal lines.
\end{enumerate}
\end{corollary}

\begin{proof}[Proof of Corollary \ref{cor:linengcn}]
The statements of i) concerning a) and b) follow from Lemma \ref{lem:CG1} and Lemma\ref{lem:CG2}, respectively.

For the statement ii) assume that $M$ is a maximal line intersecting
${\ell}$ at a node $A$ and there is no other maximal line passing through
that node.

\noindent Now suppose that $|\Xset_{\ell}| = \binom{n}{2}.$ Then, in view of Theorem \ref{thm:linengcn}, there is a maximal line $M_0$ such that $M_0 \cap {\ell} \notin \Xset.$
According to Lemma \ref{lem:CG1} we have that
$\Xset_{\ell} = \Xset \setminus ({\ell} \cup M_0).$
Therefore we get
$|M\cap \Xset_{\ell}| = |M\cap [\Xset\setminus ({\ell} \cup M_0)]| = (n+1)-2=n-1,$
since $M$ intersects ${\ell}$ and $M_0$ at two distinct nodes.

Next suppose that $|\Xset_{\ell}| = \binom{n-1}{2}.$ Then there are two
maximal lines $M'$ and $M''$ such that $M'\cap M''\cap {\ell} \in \Xset.$
Now, according to Lemma \ref{lem:CG2}, we have that $\Xset_{\ell} = \Xset
\setminus ({\ell} \cup M'\cup M'').$ Therefore we get $|M\cap \Xset_{\ell}| =
|M\cap [\Xset\setminus ({\ell} \cup M'\cup M'')]| = (n+1)-3=n-2,$ since
$M$ intersects ${\ell}, M'$ and $M''$ at three distinct nodes.
\end{proof}

\begin{remark}\label{rem:linengcn}
Assume that the conditions of Theorem \ref{thm:linengcn} take place
and $\Xset_{\ell}\neq\emptyset.$ Assume also that $M$ is a maximal line
of $\Xset$ such that $M$ intersects ${\ell}$ at a node and no node from
$M$ uses ${\ell}.$ Then there is another maximal line $M'$ such that
$M\cap M'\cap {\ell} \in \Xset$ and therefore no node from $M'$ uses ${\ell}$ either.
\end{remark}

\subsection{The proof of Theorem \ref{thm:linengcn}\label{s:pnlgcn}}

Let us start with
\subsubsection{The case $n=1$\label{ss:case1}} $GC_1$
sets consist of $3$ non-collinear nodes. Consider a such set
$\Xset=\left\{A,B,C\right\}$ and an $1$-node line ${\ell}$ that passes,
say, through $A.$ We have that no $1$-node line is used in $GC_n$
sets. Thus $\Xset_{\ell}=\emptyset.$ Therefore we may assume that both
equalities in \eqref{2bin} take place. Note also that both
implications i) and ii) of Theorem \ref{thm:linengcn} take place.
Indeed, the maximal line through $B$ and $C$ does not intersect ${\ell}$
at a node. And at the same time the other two maximal lines, i.e.,
$2$-node lines through $A, B$ and $A, C$ intersect the line ${\ell}$ at
the node $A.$

\subsubsection{The case $n=2$\label{ss:case2}}

We divide this case into $2$ parts.

\vspace{.2cm} \noindent\emph{ 1. $GC_2$ sets with $3$ maximal
lines:}\vspace{.2cm}

Consider a $GC_2$ set $\Xset$ with exactly $3$ maximal lines. These
lines intersect each other at $3$ non-collinear nodes, called
vertices. Except these $3$ nodes, there are $3$ more (non-collinear)
nodes in $\Xset$, one in each maximal line, called "free" nodes.
Here the $2$-node lines are of $2$ types:

a) $2$-node line ${\ell}$ that does not pass through a vertex. Notice
that ${\ell}$ is used only by one node and the implication i) of Theorem
holds. Namely, there is a maximal line that does not intersect ${\ell}$
at a node.

b) $2$-node line that passes through a vertex. Notice that no node
uses a such line and the implication ii) holds.

\vspace{.2cm} \noindent\emph{ 2. $GC_2$ sets with $4$ maximal
lines:}\vspace{.2cm}

In this case we have the Chung-Yao lattice (see Subsection
\ref{ss:CY}). Here all $6$ nodes of $\Xset$ are intersection nodes
of the maximal lines and the only used lines are the maximal lines.
Thus in this case any $2$-node line is not used and evidently the
implication ii) holds.

\subsubsection{The case $n=3$\label{ss:case3}}

We divide this case into $3$ parts:

\vspace{.2cm} \noindent\emph{ 1. The case of $GC_3$ sets with
exactly $3$ maximal lines:}\vspace{.2cm}

Consider a $GC_3$ set $\Xset$ with exactly $3$ maximal lines.  By
the properties of maximal lines we have that they form a triangle
and the vertices are nodes of $\Xset.$  There are $6\ (=3\times 2)$
more nodes, called "free", $2$ in each maximal line. There
is also one node outside the maximal lines, denoted by $O.$ We find
readily that the $6$ "free" nodes are located in $3$ lines passing
through $O,$ $2$ in each line (see Fig. \ref{pic2}).

\begin{figure}[ht] %stronger is !ht, strongest H
\centering
\includegraphics[scale=0.5]{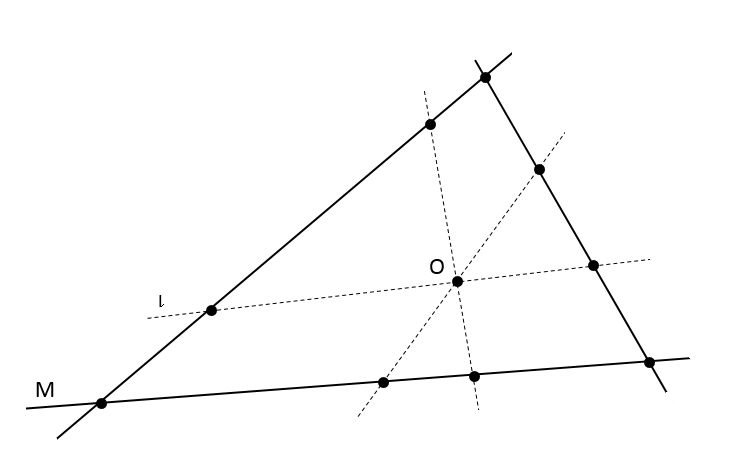}
\caption{Three $3$-node lines}\label{pic2}
\end{figure}

These $3$ lines are the only $3$-node lines in this case. We have
that for a such line ${\ell}$ there is a maximal line $M$ that does not
intersect ${\ell}$ at a node, i.e., the implication i) of Theorem holds.
Also we have that ${\ell}$ is used  by exactly $3$ nodes. Namely, by the
nodes that do not belong to ${\ell}\cup M.$

\vspace{.2cm} \noindent\emph {2. The case of $GC_3$ sets with
exactly $4$ maximal lines: }\vspace{.2cm}

Now consider a $GC_3$ set $\Xset$ with exactly $4$ maximal lines. In
this case there are $6\ \left(=\binom{4}{2}\right)$ nodes that are
intersection points of maximal lines. Also there are $4$ more nodes
in maximal lines, called "free", $1$ in each. The $4$ "free" nodes
are not collinear.

Again we have two types of $3$-node lines here.

a) $3$-node line ${\ell}$ that passes through an intersection node (see
Fig \ref{pic5}).
\begin{figure}[ht] %stronger is !ht, strongest H
\centering
\includegraphics[scale=0.5]{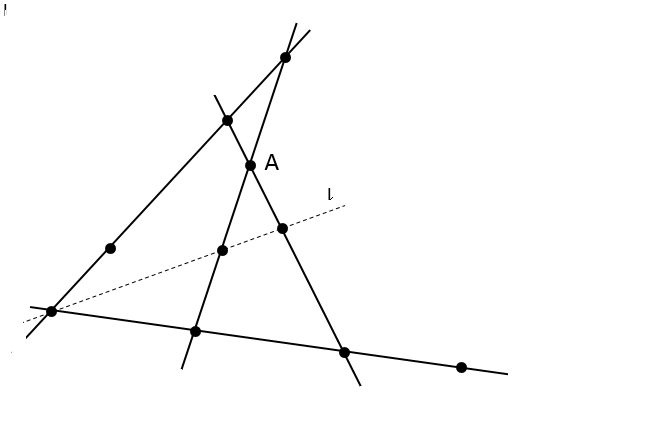}
\caption{$3$-node line passing through an intersection
node}\label{pic5}
\end{figure}

Note that a $3$-node line can pass through at most one such node. Indeed, if a line passes through two intersection nodes then it
cannot pass through any third node.

Notice that ${\ell}$ is used by only one node $A$ and the implication ii) of Theorem takes place.

b) $3$-node line ${\ell}$ that passes through $3$ "free" nodes (see Fig
\ref{pic3}).
\begin{figure}[ht] %stronger is !ht, strongest H
\centering
\includegraphics[scale=0.5]{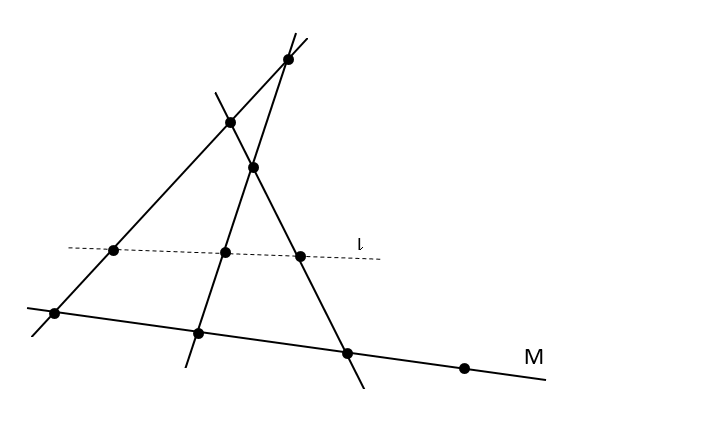}
\caption{$3$-node line through $3$ "free" nodes}\label{pic3}
\end{figure}

Notice that the maximal line $M$ whose "free" node is not lying in
${\ell}$ does not intersect ${\ell}$ at a node. Thus the implication i) holds.
In this case ${\ell}$ is used by exactly $3$ nodes. Namely, by the nodes
that do not belong to ${\ell}\cup M.$

\vspace{.2cm}

 \noindent \emph {3. The case of $GC_3$ sets with exactly
$5$ maximal lines:} \vspace{.2cm}

 In this case we have the Chung-Yao
lattice (see Subsection \ref{ss:CY}). Here all $10$ nodes of $\Xset$
are intersection nodes of $5$ maximal lines and the only used lines
are the maximal lines. Let us verify that in this case there is no
$3$-node line. Assume conversely that $\ell$ is a such line. Then
through each node there pass two maximal lines and all these maximal
lines are distinct. Therefore we get $6$ maximal lines, which is a
contradiction.

\subsection{The proof of Theorem \ref{thm:linengcn} for $n\ge 4$ \label{ss:casenge4}}

We will prove Theorem by induction on $n.$ The cases $n\le 3$ were verified.
Assume Theorem is true for all degrees less $n$ and let us prove that it is true for the degree $n,$ where $n\ge 4.$

Suppose that $|\Xset_{\ell}| \ge \binom{n-1}{2}+1.$ Then by assertion ii) of Proposition \ref{prp:linennp} we get that
there is a maximal line $M_0$ such that $M_0 \cap {\ell} \notin \Xset.$ Thus, in view of Lemma \ref{lem:CG1}, we obtain
that $|\Xset_{\ell}|=\binom{n}{2}$ and the implication i) holds.

Thus to prove Theorem it suffices to assume that

\begin{equation} \label{a}
|\Xset_{\ell}| \le \binom{n-1}{2}
\end{equation}
and to prove that the implication ii) holds, i.e., there are two maximal lines $M', M'',$ such that $M' \cap M'' \cap {\ell} \in \Xset.$ Indeed, this completes the proof in view of Lemma \ref{lem:CG2}.

First suppose that two nodes in some maximal line $M$ use the line ${\ell},$ i.e.,
\begin{equation} \label{b}
|M\cap \Xset_{\ell}|\ge 2.
\end{equation}
We have that $\Xset\setminus M$ is a $GC_{n-1}$-set. Hence, by making use of \eqref{b} and induction hypothesis, we obtain that
$$|\Xset_{\ell}| \ge |(\Xset\setminus M)_{\ell}|+2 \ge \binom{n-2}{2}+2.$$
Therefore, in view of the condition \eqref{a} and Proposition \ref{prp:linennp} iii), we conclude that
\begin{equation} \label{b'}
|\Xset_{\ell}|=\binom{n-1}{2}\ \hbox{and}\ \Nset_{\ell} \subset \beta\in \Pi_2, \ |\Nset_{\ell}|=2n.
\end{equation}
Let us use the induction hypothesis. By taking into account the first equality above and condition \eqref{b}, we obtain that
$$|(\Xset\setminus M)_{\ell}| =\binom{n-2}{2}.$$
Thus the cardinality of the set $\Nset_{\ell} \cap (\Xset\setminus M)$ equals to $2n-2\ (=2(n-1)).$
Therefore, in view of the second equality in \eqref{b'}, by using induction hypothesis we get that all the nodes in $\beta$ except possibly two are located on
two maximal lines of the set $\Xset\setminus M,$ denoted by $M'$ and $M'',$ which intersect at a node $A\in {\ell}.$
Since $n\ge 4$ each of these two maximal lines passes through at least $3$ nodes except $A,$ which belong to $\beta.$
Thus each of them divides $\beta$ and we get $\beta=M'M''.$ Finally, according to Proposition \ref{prp:linennp} iii), each of these lines passes through exactly $n$ nodes outside ${\ell}$ and therefore they are maximal also for the set $\Xset.$ Hence the implication ii) holds.

Thus we may suppose that
\begin{equation} \label{c}
|M\cap \Xset_{\ell}|\le 1 \quad \hbox{for each maximal line $M$ of the set
$\Xset.$ }
\end{equation}

Next let us verify that we may suppose that
\begin{equation} \label{e}
|M\cap \Xset_{\ell}| =1\quad \hbox{for each maximal line $M$ of the set
$\Xset.$ }
\end{equation}

Indeed, suppose by way of contradiction that no node, say in a maximal line $M_1$ uses the line
${\ell}.$
Now, in view of Theorem \ref{thm:CG}, consider two other maximal lines of $\Xset$ and denote them by $M_i,\ i=2,3.$

In view of the condition \eqref{a} and Lemma \ref{lem:CG1} we have that there is no maximal line $M_0$ such that $M_0\cap {\ell}\notin \Xset,$ i.e., all the maximal lines of $\Xset$ intersect the line ${\ell}$ at a node of $\Xset.$ Then as was mentioned above, if there are two maximal lines
intersecting at a node in ${\ell}$ then Theorem follows from Lemma \ref{lem:CG2}.

Thus, we may suppose that the $3$ maximal lines $M_i,\ i=1,2,3,$ intersect the line ${\ell}$ at $3$ distinct nodes, denoted by $C_i,\ i=1,2,3,$ respectively.

Then
consider the $GC_{n-1}$-set $\Xset_2:= \Xset\setminus M_2.$ We may
assume that $(\Xset_2)_{\ell}\neq\emptyset.$ Indeed, otherwise by
induction hypothesis and \eqref{2bin} we would obtain that $n-1=2,$ i.e, $n=3.$ In
$\Xset_2$ no node of the maximal line $M_1$ uses ${\ell}.$ By induction
hypothesis, in view of Remark \ref{rem:linengcn}, we have that there
is a maximal line $M'_1$ of this set intersecting ${\ell}$ at $C_1.$ In the same
way we get that there is a maximal line $M''_1$ in the set
$\Xset\setminus M_3$ intersecting ${\ell}$ at $C_1.$ Now if the maximal
line $M'_1$ coincides with $M''_1$ then we get readily that it is
maximal also for $\Xset$ which completes the proof in view of Lemma
\ref{lem:CG2}. Thus suppose that the maximal lines $M'_1$ and $M''_1$
are distinct. Then consider the $GC_{n-2}$-set $\Xset\setminus
(M_2\cup M_3).$ Here we have $3$ maximal lines $M_1, M'_1$ and
$M''_1$ intersecting at the node $C_1,$ which is a contradiction.

\noindent Thus we have that \eqref{e} holds, i.e., there is only one node in each maximal line $M_i,\ i=1,2,3,$ using the line ${\ell}.$
Notice that at most one node can be intersection node of these $3$ maximal lines, since otherwise we would have $2$ nodes in a maximal line that use ${\ell}.$ Consider a node $A$ which lies, say in $M_3,$
uses ${\ell}$ and is not an intersection node, i.e., does not lie in the maximal lines $M_1$ and $M_2$ (see Fig. \ref{pic4}).

Consider the $GC_{n-1}$ node set $\Xset_i:=\Xset\setminus M_i$ for any fixed $i=1,2.$ In the maximal line $M_3$
there is only one node using ${\ell}.$ Therefore, in view of the induction hypothesis and Corollary \ref{cor:linengcn}, we have that
\begin{equation}\label{d}
|(\Xset_i)_{\ell}| = 1,\ i=1,2.
\end{equation}
We may conclude from here that there is only one node in $M_1\cup M_2,$ namely the intersection node $B:=M_1\cap M_2,$ that uses the line ${\ell}.$

At the same time we get from \eqref{d} also that $(n-1)=2,$ or $(n-1)-1=2.$
Therefore $n\le 4,$ i.e., we may assume that $n=4.$

\subsubsection{A special case}

Thus it remains to consider the case $n=4$ with $|\Xset_{\ell}| = 2.$
Recall that one of the nodes: $A$ belongs to only one maximal line
$M_3.$ While the other node: $B$ is the intersection node of the
maximal lines $M_1$ and $M_2$ (see Fig. \ref{pic4}).
\begin{figure}[ht] %stronger is !ht, strongest H
\centering
\includegraphics[scale=0.5]{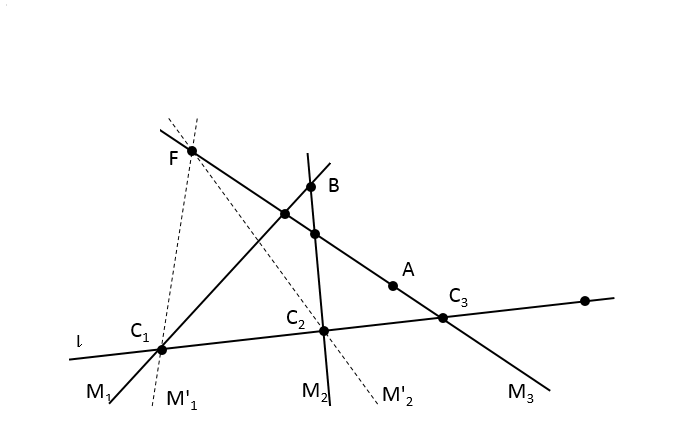}
\caption{A special case}\label{pic4}
\end{figure}
We will show that this case is not possible.

Consider the $GC_3$-set $\Xset_1:=\Xset\setminus M_1.$ The line ${\ell}$
is used by one node here: $A$ and no node in maximal line $M_2$
uses it. Thus we conclude
that there is a maximal $M_2'$ passing through $C_2.$

Now, denote by $E,$ the intersection node of the maximal lines
$M_2'$ and $M_3.$ Let us identify this node among the $5$ nodes in
$M_3.$ Notice that evidently $E$ is different from $C_3$ - the intersection node with  ${\ell}.$

We have that $E$ is different also from the intersection nodes with  $M_1$ or with $M_2.$ Indeed, three
maximal lines cannot intersect at a node.

Finally note that $E$ is different also from the node $A,$ since it uses ${\ell}$ and therefore it
does not belong to $M_2'.$

Thus $E$ coincides necessarily with the fifth node in $M_3$ denoted by $F.$

Now consider the $GC_3$-set $\Xset_2:=\Xset\setminus M_2.$ Again the line ${\ell}$
is used by one node here: $A$ and no node in maximal line $M_1$
uses it. Thus we conclude
that there is a maximal $M_1'$ passing through $C_1.$

Then, exactly in the same way as above, we may conclude that $M_1'$
intersects $M_3$ at $F.$

Finally, consider the $GC_2$-set $\Yset:=\Xset\setminus (M_1\cup M_2).$
Notice that the lines $M_1', M_2'$ and $M_3$ are $3$ maximal lines intersecting at the node $F,$
which is a contradiction. $\square$

\begin{remark}\label{rem:nle5}
Let us mention that in the cases $n\le 5$ Theorem \ref{thm:linengcn}
is valid without the assumption concerning the Gasca-Maeztu
conjecture.
\end{remark}

% Set the ending of a LaTeX document

\subsection{A conjecture concerning  $GC_n$ sets \label{s:conj}}
\begin{conjecture}\label{mainc}
Assume that Conjecture \ref{conj:GM} holds for all degrees up to $\nu$. Let $\Xset$ be a $GC_n$ set, $n \leq \nu$ and
${\ell}$ be a line passing through exactly $k$ nodes of $\Xset$ set. Then we have that
\begin{equation} \label{1bin} |\Xset_{\ell}| = \binom{s}{2},\ \hbox{for some} \ 2k-n-1\le s \le k.
\end{equation}
Moreover, for any maximal line $M$ of $\Xset$ we have:
\begin{enumerate}
\item
$|M\cap \Xset_{\ell}| = 0$ if

$M \cap {\ell} \notin \Xset$  or if

there is another maximal line $M'$ such that $M\cap M'\cap {\ell} \in \Xset;$

\item
$|M\cap \Xset_{\ell}| = s-1\ $
if $\ \binom{s}{2}=|\Xset_{\ell}|, \ \ \hbox{where}\ \ 2k-n-1\le s \le k,$\\
for all the remaining maximal lines.
\end{enumerate}
\end{conjecture}

\vspace{3mm}

%%%%%%%%%%%%%%%%%%%%%%%%%%%%%%%%%%%%%%%%%%%%%%%%%%%%%%%%%%%%%%%%%%%%%%%%
%%%%%%%%%%%%%%%%%%%%%%%%%%%%%%%%%%%%%%%%%%%%%%%%%%%%%%%%%%%%%%%%%%%%%%%%
\noindent Vahagn Bayramyan, Hakop Hakopian \vspace{2mm}

\noindent{Department of Informatics and Applied Mathematics\\
Yerevan State University\\
A. Manukyan St. 1\\
0025 Yerevan, Armenia}

\vspace{1mm}

\noindent E-mails: vahagn.bayramyan@gmail.com, hakop@ysu.am

\end{document}